\documentclass{article}

\usepackage{amsfonts,amsmath,amsthm,amssymb}
\usepackage{epsfig}
\usepackage{enumerate}
\usepackage{makeidx}
\usepackage{color}
\usepackage{psfrag}

\usepackage[top=2in, bottom=1.5in, left=1in, right=1in]{geometry}

\newtheorem{thm}{Theorem}

\newtheorem{lem}[thm]{Lemma}
\newcommand{\G}{\mathcal{G}}
\newcommand{\B}{\mathcal{B}}
\newcommand{\D}{\Delta}
\newcommand{\ex}{\mathrm{Ex}}
\newcommand{\ds}{\displaystyle}

\title{Maximum degree in minor-closed classes of graphs}
\author{Omer Gim\'{e}nez~\thanks{Google Inc., Palo Alto. E-mail: \texttt{omer.gimenez@gmail.com}}
\and
Dieter Mitsche~\thanks{Laboratoire J.A. Dieudonn\'{e}, Universit\'{e} de Nice-Sophia Antipolis, France. E-mail: \texttt{dmitsche@unice.fr}}
\and
Marc Noy~\thanks{Departament
de Matem\`{a}tica Aplicada II, Universitat Polit\`{e}cnica de Catalunya,
Barcelona, Spain. E-mail: \texttt{marc.noy@upc.edu}}}
\date{}

\begin{document}

\maketitle

\begin{abstract}
Given a class of graphs $\G$ closed under taking minors, we study the maximum degree $\Delta_n$ of random graphs from $\G$ with $n$ vertices. We prove several lower and upper bounds that hold with high probability. Among other results, we find classes of graphs providing  orders of magnitude for $\Delta_n$ not observed before, such us $\log n/ \log \log \log n$ and
$\log n/ \log \log \log \log n$.

\end{abstract}

\section{Introduction}
\thispagestyle{empty}

A class of labelled graphs $\mathcal{G}$ is minor-closed if whenever a graph $G$ is in
$\mathcal{G}$ and $H$ is a minor of $G$, then $H$ is also in $\mathcal{G}$.
A basic example is the class of planar graphs or, more generally, the class of graphs embeddable in a fixed surface.

All graphs in this paper are labelled. Let $\mathcal{G}_n$ be the graphs in $\G$ with $n$ vertices. By a random graph from $\G$ of size $n$ we mean a graph drawn with uniform probability from $\G_n$. We say that an event $A$ in the class $\G$ holds with high probability (w.h.p.) if the probability that $A$ holds in $\G_n$ tends to 1 as $n\to\infty$. Let $\D_n$ be the random variable equal to the maximum vertex degree in random graphs from $\G_n$. We are interested in events of the form
$$
    \D_n \le f(n) \qquad \hbox{w.h.p.}
$$
and of the form
$$
    \D_n \ge f(n) \qquad \hbox{w.h.p.}
$$
Typically $f(n)$ will be of the form $c \log n$ for some constant $c$, or some related functions. Throughout this paper $\log n$ refers to the natural logarithm.

A classical result says that for labelled trees $\D_n$ is of order $\log n / \log \log n$ (see~\cite{moon}). In fact, much more precise results are known in this case, in particular that (see~\cite{schmutz})
$$
    {\D_n \over \log n/ \log\log n} \to 1    \qquad \hbox{in probability}.
$$
McDiarmid and Reed~\cite{MR} show  that for the class of planar graphs there exist constants $0<c_1<c_2$ such that
$$
    c_1 \log n < \D_n < c_2 \log n   \qquad \hbox{w.h.p.}
$$
More recently this result has been strengthened using subtle analytic and probabilistic methods~\cite{DGNPS}, by showing the existence of  a computable constant $c$ such that
$$
    {\D_n \over \log n} \to c    \qquad \hbox{in probability}.
$$
Analogous results have been proved for series-parallel and outerplanar graphs~\cite{DGN3}, with suitable constants.
For planar maps (planar graphs with a given embedding), much more precise results are known~\cite{GW} on the distribution of~$\Delta_n$.

The goal in this paper is to analyze the maximum degree in additional minor-closed classes of graphs. Our main inspiration comes from the work of McDiarmid and Reed mentioned above. The authors develop proof techniques based on double counting, that assume only mild conditions on the classes of graphs involved. We know explain the basic principle.

Let $\G$ be a class of graphs and suppose we want to show that a property $P$ holds in $\G$ w.h.p. Let $\B_n$ the graphs in $\G_n$ that do not satisfy $P$ (the `bad' graphs). Suppose that for every graph in $\B_n$ we have a rule producing at least $C(n)$ graphs in $\G_n$ (the `construction' function). A graph in $\G_n$ can be produced more than once, but assume every graph in $\G_n$ is produced at most $R(n)$ times (the `repetition' function). By double counting we have
$$
        |\B_n| C(n) \le |\G_n| R(n),
$$
hence
$$
    {|\B_n| \over |\G_n| } \le {R(n) \over C(n)}.
$$
If the procedure is such that $C(n)$ grows faster than $R(n)$, that is $R(n) = o(C(n))$, then we conclude that $|\B_n| =o(|\G_n|)$, that is the proportion of bad graphs goes to 0. Equivalently, property $P$ holds w.h.p. We often use the equivalent formulation $C(n)/R(n) \to \infty$.

We will apply this principle in order to obtain lower and upper bounds on the maximum degree for several classes. In this context, lower bounds are easier to obtain, and only in some cases we are able to prove matching upper bounds. The proof of the upper bound for planar graphs in~\cite{MR} depends very strongly on planarity and seems difficult to adapt it to general situations; however we obtain such a proof for outerplanar graphs. On the other hand, we develop new tools for proving upper bounds based on the decomposition of a connected graph into 2-connected components.

Here is a summary of our main results. We denote by $\ex(H)$ the class of graphs not containing $H$ as a minor.
All the claims hold w.h.p. in the corresponding class, and $c,c_1$ and $c_2$ are suitable positive constants.
The fan graph $F_n$ consists  of a path with $n-1$ vertices plus a vertex adjacent to all the vertices in the path.

\begin{itemize}
\item In $\ex(C_4)$ we have, for all $\epsilon >0$, $$
    (2-\epsilon) {\log n  \over \log\log n} \le \Delta_n \le (2+\epsilon) {\log n \over \log\log n}.
$$
 \item In $\ex(C_5)$ we have, for all $\epsilon >0$, $$
    (1-\epsilon) {\log n  \over \log\log\log n} \le \Delta_n \le (1+\epsilon) {\log n \over \log\log\log n}.
$$
  \item In $\ex(C_6)$ we have
$$
    c_1 {\log n  \over \log\log\log n} \le \Delta_n \le c_2 {\log n \over \log\log\log n}.
$$
  \item In $\ex(C_7)$ we have
$$
    c_1 {\log n  \over \log\log\log\log n} \le \Delta_n \le c_2 {\log n \over \log\log\log\log n}.
$$
  \item If $H$ is 2-connected and contains $C_{2\ell+1}$ as a minor, then in $\ex(H)$ we have
$$
    \Delta_n \ge c {\log n \over \log^{(\ell+1)}n},
$$
where $\log^{(\ell+1)} n=\log \cdots \log n$, iterated $\ell+1$ times.


  \item If $H$ is 2-connected and is not a minor of $F_n$ for any $n$,
     then in $\ex(H)$ we have
$$
    \Delta_n \ge c \log n.
$$
\item For the class of outerplanar graphs, we have
$$
    c_1 \log n \le \Delta_n \le c_2 \log n.
$$
This result was proved in a stronger form using analytic methods in~\cite{DGN3}.
\end{itemize}

\noindent
The results on $\ex(H)$ also hold when forbidding more than one graphs as a minor, as discussed in the next section.
The plan of the paper is as follows. In Section~\ref{sec:lower} we prove the lower bounds for the maximum degree. In Section~\ref{sec:blocks} we determine the structure of 2-connected graphs in the classes $\ex(C_5)$, $\ex(C_6)$ and $\ex(C_7)$.
This is needed in the proofs for the upper bounds, which are contained in Section~\ref{sec:upper}. We conclude with some remarks and several conjectures and open problems.

\section{Lower bounds}\label{sec:lower}

A \emph{pendant} vertex is a vertex of degree one. The following lemma follows from~\cite{colin}.

\begin{lem}\label{lem:isolated}
Let $H_1,\dots,H_k$ be 2-connected graphs and let $\mathcal{G}=\ex(H_1,\dots,H_k)$. Then there is a constant $\alpha >0$ such that a graph in $\mathcal{G}_n$ contains at least $\alpha n$ pendant vertices w.h.p.
\end{lem}

To illustrate our proof technique, we reprove the following well-known result  (see~\cite{moon}, and see~\cite{schmutz} for more precise results, as mentioned above), but without the need of enumerative tools.

\begin{lem}\label{lem:c3lower}\cite{moon,schmutz}
In the class of trees, for every $\epsilon > 0$,
$$
    (1-\epsilon) {\log n  \over \log\log n} \le \Delta_n.
$$
\end{lem}
\begin{proof}
 Let $\mathcal{G}$ be the class of trees, and $\G_n$ the class of trees with exactly $n$ vertices. Let $\B_n \subseteq \G_n$ denote the set of bad graphs with  $\Delta_n < (1-\epsilon) {\log n  \over \log\log n}$, and suppose for contradiction that $|\mathcal{B}_n| \geq \mu |\mathcal{G}_n|$ for some $\mu > 0$, infinitely often. Our goal is to show that we can obtain $\omega(|\mathcal{B}_n|)$ new graphs in $\mathcal{G}_n$, or equivalently, $C(n) / R(n) \rightarrow \infty$, contradicting $|\mathcal{B}_n| \geq \mu |\mathcal{G}_n|$. Let $G$ be a graph in $\B_n$. By Lemma~\ref{lem:isolated}, $G$ has at least $\alpha n$ pendant vertices w.h.p. Choose from these a subset of size $s+1$, where $s=\lceil (1-\epsilon) {\log n \over \log \log n} \rceil$, and delete all their pendant edges. Among those choose a vertex, call it $v_1$, and make it adjacent to all other $s$ vertices. Finally, choose a vertex $u$ different from the $s+1$ chosen vertices, and make $u$ adjacent to $v_1$ (we have at least $n-s \geq n/2$ choices for $u$). In this way one can construct at least $\binom{\alpha n}{s+1}(s+1) \frac{n}{2}$ graphs.
 From how many graphs $G$ may the newly constructed graph $G'$ come from? We identify $v_1$ as the only vertex with largest degree in $G'$ and $u$ as the only non-pendant neighbor of $v_1$. In order to reconstruct $G$ completely we only need  to reattach the $s+1$ vertices in all possible ways, which can be done in at  most $n^{s+1}$ ways. Hence
$$
    {C(n) \over R(n) } \geq \frac{\binom{\alpha n}{s+1}(s+1)n}{2n^{s+1}} \geq \frac{n (\alpha/2)^{s+1}}{2s!}.
$$
Taking logarithms, this gives
$$
    \log {C(n) \over R(n) } \geq \log n  - s \log s - O(s) = \log n - (1-\epsilon)(1-o(1))\log n,
$$
which tends to infinity. Hence, $|\mathcal{B}_n|=o(|\mathcal{G}_n|)$, as was to be proved.
\end{proof}

\noindent Notice that exactly the same proof works
for the class of forests, that is, $\ex(C_3)$.

\medskip
Now we are ready to state new results that can be obtained using our techniques.
In order to prove a lower bound for $\Delta_n$ in a class $\G$, the basic idea is to generalize the previous proof. Take a graph $G$ in    $\G_n$ whose maximum degree is too small (a bad graph), take enough pendant vertices and make with them a special graph $S$ rooted at a special vertex $v$ (in the previous proof a star rooted at its center), and attach $S$ to $G$ through a single edge, producing a new graph $G'$ in $\G_n$. Then $v$ becomes the unique vertex of maximum degree $s=|S|$, and $G$ can be reconstructed from $G'$ easily by reattaching the vertices in $S$, which are neighbours of $v$ in $G'$. 
Double counting is then used to show that the proportion of bad graphs 
goes to~$0$ as $n$ goes to infinity.  

\begin{thm}\label{th:lower}
The following claims refer to the class $\ex(H_1,\dots,H_k)$, where $c>0$ is a suitable constant.
\begin{enumerate}
  \item If all the $H_i$ are 2-connected and none of them is a minor of a fan graph $F_n$, then
$$
    \D_n \ge c \log n      \qquad \hbox{w.h.p.}
$$
This holds in particular if the $H_i$ are 3-connected or not outerplanar.

 \item If all the $H_i$ are 2-connected and contain $C_{4}$ as a minor, then for every $\epsilon >0$,
$$
    \D_n \ge (2-\epsilon) {\log n \over \log\log  n}     \qquad \hbox{w.h.p.}
$$

 \item If all the $H_i$ are 2-connected and contain $C_{5}$ as a minor, then for every $\epsilon >0$,
$$
    \D_n \ge (1-\epsilon) {\log n \over \log\log \log  n}     \qquad \hbox{w.h.p.}
$$

 \item If all the $H_i$ are 2-connected and contain $C_{2\ell+1}$ as a minor for some $\ell \geq 3$, then
$$
    \D_n \ge c {\log n \over \log^{(\ell+1)} n}     \qquad \hbox{w.h.p.}
$$
In particular,  the bound $\D_n \ge c \log n / \log \log n$ always holds, since every 2-connected graph contains $C_3$ as a minor.
\end{enumerate}
\end{thm}

\begin{proof}
Throughout the proof we will assume for contradiction that there is some constant $\mu > 0$ such that for each item and its corresponding graphs in $\mathcal{B}_n$, we have $|\mathcal{B}_n| \geq \mu |\mathcal{G}_n|$ infinitely often. Our goal is to show that we can obtain $\omega(|\mathcal{B}_n|)$ new graphs in $\mathcal{G}_n$, or equivalently, $C(n) / R(n) \rightarrow \infty$, contradicting $|\mathcal{B}_n| \geq \mu |\mathcal{G}_n|$.
Since, by assumption, $|\mathcal{B}_n| \geq \mu |\mathcal{G}_n|$, we know by Lemma~\ref{lem:isolated} that for $n$ large enough almost all graphs in $\mathcal{B}_n$ have at least $\alpha n$ pendant vertices, and we will condition on this from now on.

\textbf{1.}  Let $\mathcal{G}= \ex(H_1,\dots,H_k)$ and let $\mathcal{B}_n \subseteq \mathcal{G}_n$ the graphs with $\Delta_n < c \log n$, where $c$ is a sufficiently small constant, and let $h = \lceil c \log n\rceil$. 
Let $G$ be a graph in $\B_n$. Choose an ordered list $v_1,\dots, v_h$ of $h$ pendant vertices in $G$, delete the edges joining the $v_i$ to the rest of the graph, and make a copy of $F_h$ with a path $v_2,\dots,v_h$ and $v_1$ adjacent to all of them. Select a vertex $u$ of $G$ different from the $v_i$ and make it adjacent to $v_1$.
The graph $G'$ constructed in this way belongs to $\G_n$, since the $H_i$ are 2-connected and none of them is a minor of a fan graph, and has the same number of vertices as $G$.

The number of graphs constructed in this way is at least (where $(m)_k$ denotes a falling factorial)
$$
    (\alpha n)_h (n-h) \ge  \left({\alpha n\over 2}\right)^h n,
$$
the last inequality being true for $n$ large enough; we use the fact that $h = \lceil c \log n \rceil$ is small compared with~$n$.

How many times a graph $G'$ can be constructed in this way? Since $G \in \B_n$,  $v_1$ can be identified as the only vertex of  degree $h$. Vertices $v_2,\dots,v_n$ can be identified as the neighbors of $v_1$ inducing a path (among the neighbors of $v_1$, $u$ is the only cut-vertex, and hence it can be identified easily). In order to recover $G$, we delete all the edges among the $v_i$ and the edge $v_1u$, and make $v_1,\dots,v_n$ adjacent to one of the remaining vertices through a single edge. The number of possibilities is at most
$$
    (n-h)^h \le n^h.
$$
Summarizing, we can take $C(n) =  (\alpha/2)^h n^{h+1}$
and $R(n) = n^h$. Then
$$
    {C(n) \over R(n) } \ge n  (\alpha/2)^{c \log n},
$$
which tends to infinity if $c$ is small enough.
This finishes the proof.

\textbf{2.}
Assume now that the $H_i$ contain $C_4$ as a minor, that is, they all contain a cycle of length at least four. As before, let $\mathcal{G}= \ex(H_1,\dots,H_k)$, let $\mathcal{B}_n \subseteq \mathcal{G}_n$ be the graphs with $\Delta_n < (2-\epsilon) \log n/\log\log n$, and let  $s = \lceil  (2-\epsilon) \log n/\log\log n\rceil$.
Let $G$ be a graph in $\B_n$. Choose an (unordered)  set of $s+1$ pendant vertices $v_1,\dots, v_{s+1}$ in $G$, and delete the edges joining the $v_i$ to the rest of the graph. Among those choose one of them, say $v_1$, and make it adjacent to all others. The other $s$ vertices are paired up, and vertices of pairs are made adjacent (assume $s$ even, otherwise one vertex remains unpaired). Finally, another pendant vertex $u$ is chosen and made adjacent to $v_1$. Note that there are at least $\alpha n/2$ choices for~$u$. There are thus at least $\binom{\alpha n}{s+1}(s+1)\left((s-1)!!\right) (\alpha n/2)$ constructions, where $(2k-1)!!=1\cdot3 \cdots (2k-1)$.  The graph $G'$ constructed in this way belongs to $\G_n$, and has the same number of vertices as $G$.
When reconstructing~$G$, $v_1$ can be  identified as the unique vertex of maximum degree, and $u$ is identified as  the only neighbor of $v_1$ adjacent to a vertex who is not a neighbor of $v_1$. Thus, only the $s+1$ chosen vertices have to be reattached, and there are at most $n^{s+1}$ choices. Hence,
$$
    {C(n) \over R(n) } \ge   \frac{\binom{\alpha n}{s+1}(\frac12 \alpha n)\left( (s+1)!!\right)}{n^{s+1}} \ge \frac{(\frac12 \alpha)^{s+2}\left((s+1)!!\right)n}{(s+1)!}.
$$
Using  $(2g-1)!!=(2g)!/(2^{g}g!)$ and taking logarithms we obtain
$$
\log {C(n) \over R(n) } \ge \log n - (s/2) \log s - O(s) = \log n - (1-(\epsilon/2))(1-o(1)) \log n,
$$
which tends to infinity, as desired.

\textbf{3.}
Now we may assume that the $H_i$ contain $C_5$ as a minor. As before, Let $\mathcal{G}= \ex(H_1,\dots,H_k)$ and let $\mathcal{B}_n \subseteq \mathcal{G}_n$ be the graphs with $\Delta_n < (1-\epsilon) \log n/\log\log\log n$, and let  $s = \lceil  (1-\epsilon) \log n/\log\log\log n\rceil$.

Let $F_{n,m}$ be the following graph. Take $m$ disjoint copies of $K_{2,n-1}^+$ (the complete bipartite graph $K_{2,n-1}$
plus an edge joining the two vertices in the part of size two),
and glue them identifying a vertex of degree $n-1$ in each copy. Notice that the longest cycle in $F_{n,m}$ is $C_4$, and that it has $mn+1$ vertices.

Let $G$ be a graph in $\B_n$. Choose a set of $s+1$ pendant vertices $v_1,\dots, v_{s+1}$ in $G$, delete the edges joining the $v_i$ to the rest of the graph, and make a copy of $F_{r,s/r}$ with the $v_i$, where $r$ is determined later. Let $v_1$ be the vertex chosen to be adjacent to all other $v_i$ (there are $s+1$ choices for this vertex). Select a vertex $u$ of $G$ different from the $v_i$ and make it adjacent to $v_1$.
The graph $G'$ constructed in this way belongs to $\G_n$, since the $H_i$ are 2-connected and have no cycle of length more than four, and has the same number of vertices as $G$.

The number of graphs constructed in this way is at least
$$
{ {\alpha n \choose s+1} (s+1) {s \choose r,\dots,r} r^{s/r} {n \over 2}
\over (s/r)!
},
$$
where the first binomial is for the choice of the pendant vertices; $(s+1)$ is for the choice of the center vertex ($v_1$), the multinomial coefficient divided by $(s/r)!$ stands for the number of partitions of the $s$ vertices into groups of size $r$; the factor $r^{s/r}$ for choice of the vertices of degree $r$ in each group; and finally $n/2$ is a lower bound for the choices of the target vertex~$u$.
The number of ways such a graph $G'$ can be constructed is at most $n^{s+1}$, the argument is the same as before. Therefore, for $n$ large enough, we have
$$
    {C(n) \over R(n) } \ge
{ {\alpha n \choose s+1} (s+1) {s \choose r,\dots,r} r^{s/r} {n \over 2}
\over (s/r)!n^{s+1}} \ge
{ ({\alpha\over2})^{s+1} {n \over 2} r^{s/r} \over (r!)^{s/r}(s/r)!
}.
$$
Here and in the following we ignore lower-order terms, and in addition equalities and inequalities have to be understood up to $1+o(1)$ terms.
Taking logarithms in the last expression we obtain the following quantity
$$
    (s+1)\log {\alpha\over2} + \log{1\over2} + \log n  +
    {s \over r} \log r - s \log r - {s\over r} \log {s \over r}.
$$
For the choices
$$
    s = (1-\epsilon){\log n \over \log\log\log n} , \qquad
    r = {2 \log s \over \epsilon \log\log s}
$$
we can safely ignore the term $(s+1)\log(\alpha/2)+\log(1/2)+(s/r)\log r$.
Plugging in these values of $s$ and $r$ into the remaining term, we obtain
\begin{align*}
& \log n - s \log r - \frac{s}{r}\log \frac{s}{r}  \\
\geq & \log n - s (\log \log s - \log \log \log s)- \frac{\epsilon}{2}s \log \log s \\
\geq & \log n - \left(1+\frac{\epsilon}{2}\right) s \log \log s \\
\geq & \log n - \left(1+\frac{\epsilon}{2}\right)(1-\epsilon)\log n,
\end{align*}
which tends to infinity, since $(1+\frac{\epsilon}{2})(1-\epsilon) < 1$.

\textbf{4.}
As before, assume that the $H_i$ contain $C_{2\ell+1}$ as a minor, and let $\mathcal{G}= \ex(H_1,\dots,H_k)$. Let $\mathcal{B}_n \subseteq \mathcal{G}_n$ be the graphs with $\Delta_n < c \log n/\log^{(\ell+1)}n$ (where $c$ is a small enough constant), and let  $s = \lceil c \log n/\log^{(\ell+1)}n \rceil$. 

Let $G$ be a graph in $\B_n$. Choose a set of $s+1$ pendant vertices $v_1,\dots, v_{s+1}$ in $G$, delete the edges joining the $v_i$ to the rest of the graph, and make a copy of the following graph $F$ with the $v_i$. First, as before, choose one special vertex, call it $v_1$, and make it adjacent to all other $v_i$.
Group the remaining $v_i$ (all except for $v_1$) into groups of size $r_1=\log s /\log^{(\ell)}s$
(we ignore rounding issues, taking care of them below). Choose in each of the $s/r_1$ groups a center vertex. Call all center vertices to be vertices at level $1$.
Iteratively, for $i=1,\ldots,\ell-2$, do the following: group each group of size $r_i-1$ (from each group we eliminate the center vertices at level $i$) into subgroups of size $r_{i+1}=\log^{(i+1)} s/\log^{(\ell)}s$. Choose in each subgroup a new center vertex, and call all center vertices chosen in this step to be vertices at level $i+1$.  Connect each center vertex at level $i$ with all center vertices at level $i+1$ resulting from subgroups of the group of vertex~$i$. Connect all center vertices at level $\ell-1$ with the remaining vertices of its corresponding subgroup  (those vertices not chosen as centers).

Observe that the graph $F$ does not contain a $C_{2\ell+1}$, since in the construction we add a forest of maximum path length $2(\ell-1)$ to a star centered at $v_1$, and thus the maximum cycle length is~$2\ell$.

Next select a vertex $u$ of $G$ different from the $v_i$ and make it adjacent to $v_1$. The graph $G'$ constructed in this way belongs to $\G_n$, and has the same number of vertices as $G$.
As before, we count the number of different graphs obtained by applying this construction to one graph of $\mathcal{B}_n$. We obtain at least
$$
\frac{\frac  n 2 \binom{\alpha n}{s+1}(s+1)\binom{s}{r_1,\ldots,r_1} \prod_{i=1}^{\ell-2} \binom{r_i-1}{r_{i+1},\ldots,r_{i+1}}^{s/r_i(1+\beta_i)} (r_i-1)^{s/r_i(1+\beta_i)}}{(\frac{s}{r_1})! \prod_{i=1}^{\ell-2}((\frac{r_i-1}{r_{i+1}})!)^{s/r_i(1+\beta_i)}}
$$
many graphs, where the $\beta_i=o(1)$ take into the account rounding issues and also the fact that in the $i$th step only $r_i-1$ vertices are split into subgroups of size $r_{i+1}$ (for example, we approximate $\frac{s (r_1-1)}{r_1 r_2}$ by $\frac{s}{r_2}$; $\beta_2$ accounts for the difference. Indeed, even for the last term $\beta_{\ell-2}$ the error term is bounded from above by $\sum_{i=1}^{\ell-3}\frac{1}{r_i}=o(1)$.  By the same argument as before, a new graph can have at most $n^{s+1}$ preimages. Thus, for $n$ sufficiently large (the factors $r_i^{s/r_i}$ in the denominator are a lower bound corresponding to the fact that the factors $r_i$ in the numerator do not exactly cancel), we have 
$$
{C(n) \over R(n) } \geq
\frac{\frac12 n (\frac12 \alpha)^{s+1} ((r_1-1)!)^{s/r_1} }
{(r_1!)^{s/r_1}(\frac{s}{r_1})! (r_{\ell-1})!^{s/r_{\ell-1}(1+\beta_{\ell-1})} \prod_{i=2}^{\ell-2} r_i^{s/r_i}\prod_{i=1}^{\ell-2} ((\frac{r_i-1}{r_{i+1}})!)^{s/r_i(1+\beta_i)}}.
$$
Taking logarithms (again ignoring lower order terms), we obtain
$$
\log n+s \log (r_1-1)-s \log r_1-\frac{s}{r_1}\log \frac{s}{r_1}-s \log r_{\ell-1}-\sum_{i=1}^{\ell-2} \frac{s}{r_i}\frac{r_i-1}{r_{i+1}}\log \frac{r_i-1}{r_{i+1}}.	 $$
Using $s \log (r_1-1)=s \log(r_1)+s\log(1-1/r_1)$ and $\frac{s}{r_i}\frac{r_i-1}{r_{i+1}}\log \frac{r_i-1}{r_{i+1}} \leq \frac{s}{r_{i+1}}\log r_i$,
and once more ignoring lower order terms,
we get that this expression is at least
\begin{equation}\label{eq:expr}
\log n- \frac{s}{r_1} \log \frac{s}{r_1}-s \log r_{\ell-1}-\sum_{i=1}^{\ell-2} \frac{s}{r_{i+1}}\log r_i.
\end{equation}
Plugging in the values $r_i=\log^{(i)}s/\log^{(\ell)} s$, all but the first term are $(1+o(1))s \log^{(\ell)} s$, and thus, plugging in the value $s=c \log n/\log^{(\ell+1)}n$,  for $c < 1/\ell$, the expression in~\eqref{eq:expr} tends to infinity.
\end{proof}

\noindent{\it Remark.}
 The 2-connected graphs which are a minor of some $F_n$ consist just of a cycle and some chords, all of them incident to the same vertex. In particular, if we forbid the graph consisting of a cycle of length six $v_1v_2v_3v_4v_5v_6$ and the chords $v_1v_3$ and $v_4v_6$, the statement still holds. It also holds for the 6-cycle and the chords $v_1v_3, v_3v_5, v_5v_1$.

\section{Characterization of blocks in $\ex(C_5),\ex(C_6)$ and $\ex(C_7)$}\label{sec:blocks}

In this section we determine all 2-connected graphs in the classes $\ex(C_5),\ex(C_6)$ and $\ex(C_7)$. This is an essential ingredient for the proofs in the next section.

As usual, $K_{2,n}$ is the complete bipartite graph with partite sets of size $2$ and $n$. Denote by $K_{2,n}^+$ the graph obtained from $K_{2,n}$ by adding an edge between the two vertices of degree $n$. Then we have the following.

\begin{lem}\label{lem:c5}
The only 2-connected graphs in $\ex(C_5)$  are $K_3$,  $K_{2,m}$ and~$K^+_{2,m}$, for $m\ge2$.
\end{lem}

\begin{proof}
Let $G$ be a 2-connected graph in $\ex(C_5)$.
Since $G$ has at least four vertices, it contains  $C_4$ as a subgraph. Let $v,v_1,v_2,v_3$ be the vertices in cyclic order of a $C_4$ in $G$. Assume $v$ has degree larger than 2    and consider a neighbor $a$ of $v$ different from $v_1$ and $v_3$. Observe that $a$ cannot be adjacent to $v_1$ or $v_3$, since this would create a $C_5$. By 2-connectivity, there must exist a path from $a$ to $v_2$ containing none of $v,v_1,v_3$. Since $G$ is in $\ex(C_5)$, it follows that $a$ is adjacent to $v_2$. This holds for all neighbors of $v$. Thus,  they must form an independent set, and we obtain a copy of $K_{2,m}$. The only edge that can be added while staying in $\ex(C_5)$ is the edge $vv_2$, giving rise to $K_{2,m}^+$.
\end{proof}

Define the graph $H_{2,s,t}$ obtained by identifying a vertex $v$ of degree $s$ in $K_{2,s}$ and a vertex of degree $t$ in $K_{2,t}$, and by adding an edge between the other vertices $u$ and $w$  of degree $s$ and $t$, respectively. We denote by $H_{2,s,t}^*$ any of the following graphs: $H_{2,1,t}$ plus the edge joining $w$ and the only common neighbor of $u$ and $v$; the symmetric construction from $H_{2,s,1}$; and all the graphs obtained from $H_{2,s,t}$ or any of these by adding any of the edges $vu$ and $vw$ (see Figure~\ref{C6image}).
\begin{figure}[htpb!]
  \begin{center}
    \includegraphics[width=0.2\textwidth]{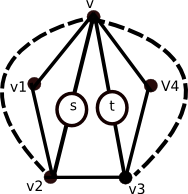}
  \end{center}
  \caption{The graph $H_{2,s,t}$ with the notation as in Lemma~\ref{lem:c6}, and with two optional edges (dashed)}\label{C6image}
\end{figure}

\begin{lem}\label{lem:c6}
The only 2-connected graphs in $\ex(C_6)$ are those in $\ex(C_5)$, the graphs $H_{2,s,t}$, and any graph of the form $H_{2,s,t}^*$, for $s,t\ge1$.
\end{lem}

\begin{proof}
Let $G$ be a 2-connected graph in $\ex(C_6)$. If $G$ is in $\ex(C_5)$, we apply the previous lemma. Otherwise let $v,v_1,v_2,v_3,v_4$ be the vertices in cyclic order of a $C_5$ in $G$ (see Figure~\ref{C6image}). Call these vertices \emph{special}. Assume $v$ has degree sufficiently large. As before, $N(v)$ is an independent set. Consider a neighbor $a$ of $v$ (excluding special vertices). As in the proof of Lemma~\ref{lem:c5}, by 2-connectivity, $a$ is adjacent to either $v_2$ or $v_3$, but not both. Let $A = N(v) \cap N(v_2) - \{v_1\}$, $B = N(v) \cap N(v_3) - \{v_4\}$, $s=|A|+1$, and $t=|B|+1$. With this notation, it can be checked that $G$ is either $H_{2,s,t}$ or is in $H_{2,s,t}^*$, possibly with $v_3$ or $v_4$ playing the role of $v$.
\end{proof}

\paragraph{Remark.} When later we refer to graphs $H_{2,s,t}$ or in $H_{2,s,t}^*$,
 with $v_3$ or $v_4$ playing the role of $v$, they will be denoted as
 $\widetilde{H}_{2,s,t}$ and $\widetilde{H}_{2,s,t}^*$.

\medskip
 Define the graph $S_{s,t,u,w}$ to be the graph constructed as follows: start with a $6$-cycle whose vertices in cyclic order are $v,v_1,v_2,v_3,v_4,v_5$, and call these vertices \textit{special}. In addition there are $w \geq 0$ vertices connecting $v_2$ and $v_4$, $s \geq 0$ vertices connecting $v$ with $v_2$, $t \geq 0$ vertices connecting $v$ with $v_4$, and $u \geq 0$ vertices connecting $v$ with both $v_2$ and $v_4$ (in all cases excluding special vertices).
 Define then by $S_{s,t,u,w}^*$ any graph obtained by possibly adding any of the edges $vv_2$, $vv_4$, $v_2v_4$, $v_1v_4$, and $v_2v_5$ (the latter edges may be added only in some cases without creating a $7$-cycle, see Figure~\ref{C7image1}).
%
\begin{figure}[htpb!]
  \begin{center}
    \includegraphics[width=0.3\textwidth]{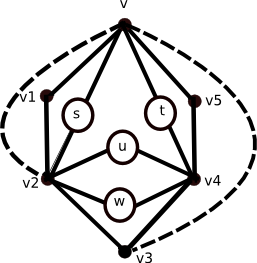}
  \end{center}
  \caption{The graph $S_{s,t,u,w}$ with the notation as in Lemma~\ref{lem:c7}, and with two optional edges (dashed)}\label{C7image1}
\end{figure}

 Finally, let $V_{s,t,E}$ the following class of graphs: start with a $6$-cycle $v,v_1,v_2,v_3,v_4,v_5$, again called special vertices. There is a set $A$ of $s \geq 0$ vertices connecting $v$ with $v_2$, and  a set $B$ of $t \geq 0$ vertices connecting $v$ with $v_4$ (always excluding special vertices).

  In addition, there is the following set of connections between $v$ and $v_3$ (not including vertices in $A$ or $B$ or special vertices) specified by $K=\{e_1,e_2,e_3,e_4,e_5,e_6\}$. There are $e_1 \geq 0$ vertices connecting $v$ with $v_3$, and $e_2$ pairs of vertices which are adjacent to each other, and both are adjacent to both $v$ and $v_3$. Furthermore, there are $e_3$ disjoint graphs $K_{2,q_i}
$ (for $i=1,\ldots,e_3$)  emanating from $v_3$, and the other vertex of degree $q_i$ is connected to $v$. For $e_4$, the construction is the same, except that for these graphs also the edge between $v_3$ and the other vertex of degree $q_i$ is present. Finally, there are $e_5$ and $e_6$ disjoint graphs $K_{2,q_i}$ which are as the graphs $e_3$ and $e_4$, but with the roles of $v_3$ and $v$ exchanged. For further reference, call the graphs of group $e_3$ and $e_4$ \emph{double stars} of degree $q_i$ emanating from $v_3$ (for $i=1,\ldots,e_3$), and those of group $e_5$ and $e_6$ double stars emanating from $v$ of degree $q_i$. All vertices appearing in any of the six groups are disjoint and we  refer to them as \emph{external} vertices. Finally, $V_{s,t,E}^*$ is the class of graphs obtained by possibly adding any of the edges $vv_2$, $vv_4$, $v_2v_4$, $v_1v_4$, $v_2v_5$ and $vv_3$ (again, some of the edges might only be added in some cases, see Figure~\ref{C7image2} for an example).
\begin{figure}[htpb!]
  \begin{center}
    \includegraphics[width=0.55\textwidth]{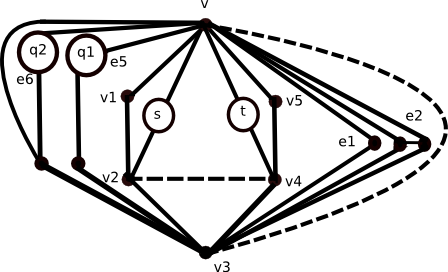}
  \end{center}
  \caption{The graph $V_{s,t,E}$ with the notation as in Lemma~\ref{lem:c7} ($e_1=e_2=1$, $e_3=e_4=0$, $e_5=e_6=1$ with corresponding degrees $q_1$ and $q_2$), and with two optional edges (dashed)}\label{C7image2}
\end{figure}

  \begin{lem}\label{lem:c7}
 The only 2-connected graphs in $\ex(C_7)$ are those in $\ex(C_6)$, the graphs  $S_{s,t,u,w}$, $V_{s,t,E}$ and the corresponding graphs $S_{s,t,u,w}^*$,  $V_{s,t,E}^*$.
 \end{lem}
 \begin{proof}
 Let $G$ be a 2-connected graph in $\ex(C_7)$. If $G$ is in $\ex(C_6)$, we apply the previous lemma.
 Otherwise let $v,v_1,v_2,v_3,v_4, v_5$ be the vertices in cyclic order of a $C_6$ in $G$, again called special. Assume  that $v$ has degree sufficiently large. Again, $N(v)$ is an independent set. We distinguish two cases now.
 In the sequel all new vertices considered are not special vertices.  \\
 \textit{Case 1:} There is no other vertex $a$ with the property that there are two internally vertex-disjoint paths of length three from $v$ to $a$. We distinguish between two subcases.\\
  \textit{Case 1.1:} Suppose first that there exist $w \geq 1$ vertices $a \in N(v)$ that are adjacent to both $v_2$ and $v_4$. Observe that the existence of such a vertex $a$ implies that no external vertex $e$ can be present in $G$, as otherwise one would have a cycle of length at least $7$ (namely, $v$,$v_1$,$v_2$,$a$,$v_4$,$v_3$,$e$,$v$). Hence, all neighbors of $v$  can be partitioned into three sets $A$, $B$ and $C$, where $A$ is the set of $s \geq 0$ vertices connected only with $v_2$, $B$ is the set of $t \geq 0$ vertices connected only  $v_4$, and $C$ is the set of  $u \geq 1$ vertices connected to both $v_2$ and $v_4$. This corresponds exactly to the graph $S_{s,t,u,w}$ with $w=0$. It is easy to check that except for edges yielding a graph in $S_{s,t,u,w}^*$, no edge can be added, as otherwise a $7$-cycle would be generated (see Figure~\ref{C7image1}).
    \\
  \textit{Case 1.2:} Suppose that there is no vertex $a \in N(v)$  adjacent to both $v_2$ and $v_4$. Let $A$ be the neighbors of $v$ connecting $v$ with $v_2$, and let $B$ be the neighbors of $v$  connecting $v$ with $v_4$.
  Let $s=|A|$ and $t=|B|$.
External vertices  connecting $v$ with $v_3$ are now possible. Note first that none of them can be adjacent to a special vertex or to a vertex in $A$ or $B$.  There can be $e_1$ vertices connecting $v$ with $v_3$, and $e_2$ pairs of vertices, adjacent to each other, both adjacent to $v$ and $v_3$. Also, we might have $e_3$ ($e_4$, respectively) double stars of degree $q_i \geq 0$ emanating from $v_3$, where the other vertex of degree $q_i$ is also adjacent to $v$ (in the case of the $e_4$ vertices, the edge between $v_3$ and the other vertex of degree $q_i$ is also present). Also, the roles of $v_3$ and $v$ can be interchanged, yielding $e_5$ double stars ($e_6$, respectively) of degree $q_i$ emanating from $v$ (in the case of the $e_6$ stars, the edge between $v$ and the other vertex of degree $q_i$ is added as well; observe that in the case of the $e_5$ double stars we may assume $q_i \geq 2$, as otherwise these vertices appear already among the $e_3$ stars). The six groups are disjoint and  there can be no  other edge between external vertices. Thus, denoting $K=\{e_1,\ldots,e_6\}$, we obtain a graph in $V_{s,t,E}$. As before, it can be checked that no other edge except for edges yielding a graph in $V_{s,t,E}^*$ can be added (see Figure~\ref{C7image2}). \\
   \textit{Case 2:} There exists at least one more vertex $a$ such that there are two internally vertex-disjoint paths of length three from $v$ to $a$.
These paths must be of the form $v, v_1, v_2, a$ and $v, v_5,v_4, a$ (if for example instead of the edge $vv_1$ there would be an edge $vz$ for some other vertex $z$, there would be a path of length $6$ going from $z$,$v$,$v_5,\ldots,v_1$, which, by 2-connectivity, would give a cycle of length at least $7$). We  suppose there are $w \geq 1$ such vertices $a$ with such paths. Observe that the existence of such a vertex $a$ implies that no external vertex $e$ can be present in $G$, as otherwise one would have a cycle of length at least $7$ (namely, the cycle $v$,$v_1$,$v_2$,$a$,$v_4$, $v_3$,$e$,$v$). All neighbors of $v$  can thus be partitioned into three sets $A$, $B$, and $C$, where $A$ are those connected only with $v_2$, $B$ those connected only with $v_4$, and $C$ those  connected both with $v_2$ and $v_4$.
We let $s=|A|,t=|B|, u=|C|$.
Let $W$  be the vertices which are neither neighbors of $v$ nor special vertices, and $w=|W|$.  Again it can be checked that they all are such that there are two internally vertex-disjoint paths of length three from $v$ to them, thus yielding a graph in $S_{s,t,u,w}$. Finally,  it can be checked that except for edges yielding a graph in $S_{s,t,u,w}^*$, no other edge can be added.
 \end{proof}

 \paragraph{Remark.} When later we refer to graphs in $S_{s,t,u,w}$ or $V_{s,t,E}$ (or to the corresponding graphs in $S_{s,t,u,w}^*$ or $V_{s,t,E}^*$), where either $v_2$, $v_3$, $v_4$ or any of the external vertices of high degree play the role of $v$, they will be denoted as
 $\widetilde{S}_{s,t,u,w}$ and $\widetilde{V}_{s,t,E}$ ($\widetilde{S}_{s,t,u,w}^*$, and $\widetilde{V}_{s,t,E}^*$, respectively).

\section{Upper bounds}\label{sec:upper}

We make repeated use of the following well-known lemma, whose proof is standard and therefore omitted.
\begin{lem}\label{lem:entropy}
Let $n_1,\ldots,n_r$ be positive integers such that $\sum_i n_i=N$ for some constant $N$. Then $\sum_i n_i \log n_i$ is minimized when all $n_i$ are equal to $N/r$.
\end{lem}

Also, we need the following lemma, whose proof is a straightforward generalization of Lemma~2.2 from~\cite{MR}.

\begin{lem}\label{lem:pendant1}
Let  $\G = \ex(H_1,\dots,H_k)$, where the $H_i$ are 2-connected. Then w.h.p. each vertex in a graph in $\G_n$ is adjacent to at most $2\log n/\log\log n$ pendant vertices.
\end{lem}

As in Section~\ref{sec:lower}, we illustrate our technique to reprove in a simpler way the following known result, complementing Lemma~\ref{lem:c3lower}.
\begin{lem}\label{lem:c3upper}\cite{moon,schmutz}
In the class of trees we have, for every $\epsilon > 0$,
$$
     \Delta_n \le (1+\epsilon) {\log n  \over \log\log n}.
$$
\end{lem}
\begin{proof}
 Let $\G$ be the class of trees and $\G_n$ the class of trees with $n$ vertices. Let $\B_n \subseteq \G_n$ now denote the set of bad graphs with  
 $$\Delta_n > (1+\epsilon) {\log n  \over \log\log n},
$$
and suppose  for contradiction that $|\mathcal{B}_n| \geq \mu |\mathcal{G}_n|$ for some $\mu > 0$, infinitely often. Let $G$ be a graph in $\B_n$, and let $v$ be a vertex with degree $k >(1+\epsilon) {\log n  \over \log\log n}.$  By Lemma~\ref{lem:isolated}, $G$ has w.h.p. at least $\alpha n$ pendant vertices, and by Lemma~\ref{lem:pendant1}, w.h.p. every vertex is adjacent to at most $2\log n/\log\log n$ pendant vertices. Hence there are w.h.p. at least $(\alpha n - 2\log n/\log\log n)\geq 2 \alpha n/3$ pendant vertices not adjacent to $v$.
Let $c=\min(\frac{\epsilon/2}{1+\epsilon}, \alpha/3)$ and choose a set of $\lceil ck \rceil$ vertices pendant vertices not adjacent to $v$ and delete their adjacent edges. Maintain vertex $v$ and delete all its adjacent edges. Attach the $\lceil ck \rceil \le 2\alpha n/3$ chosen vertices to $v$, and attach
the former $k$ neighbours of $v$ in all possible ways to any of the previously added $\lceil ck \rceil$ vertices. Observe that the new vertices have been added in a tree-like way, hence the new graph is still in $\G$. Ignoring ceilings from now on, the number of graphs constructed in this way is at least $\binom{2 \alpha n/3}{ck}(ck)^k$. From how many graphs may the newly constructed graph $G'$ come from? One has to guess vertex $v$, and then reattach the $ck$ pendant vertices, giving rise to at most $n^{ck+1}$ choices. Hence,
$$
 {C(n) \over R(n) } \geq \frac{\binom{ 2\alpha n/3}{ck}(ck)^k}{n^{ck+1}} \geq \frac{(\alpha/3)^{ck}(ck)^k}{n (ck)! }.
$$
Note that $(ck)^k / (ck)! > (ck)^{(1-c)k}$. Taking logarithms, this gives
$$
    \log {C(n) \over R(n) } \geq (1-c)k \log k - \log n - O(k) \geq (1-c)(1+\epsilon)(1-o(1))\log n - \log n,
$$
which tends to infinity by our choice of $c$. Hence, $|\mathcal{B}_n|=o(|\mathcal{G}_n|)$.
\end{proof}

Now we proceed to prove new results.
In order to prove an upper bound for $\Delta_n$ in a class $\G$, the basic idea is to generalize the previous proof. Take a graph $G$ in $\G_n$ whose maximum degree is too large (a bad graph), and let $v$ be a vertex with large degree. Consider the blocks containing $v$ and their contribution towards the degree of $v$ (in the case of trees the only blocks are single edges): the lemmas in Section \ref{sec:blocks} tell us all possible blocks that can occur. We classify the blocks  according to whether this contribution is larger or smaller than a suitable threshold. If $B$ is a block with a vertex $b$
of large degree $t$, remove the edges connecting $b$ to its neighbours $b_1,\dots,b_t$, take $ct$ pendant vertices (where $c<1$ is a suitable constant), make them adjacent to $v$, and connect arbitrarily each of the $b_i$ to any of the new $ct$ vertices. Whatever was attached to the $b_i$ remains untouched. When necessary, we add a few extra vertices and edges to ensure unique reconstruction. Blocks with small degree are not dismantled.
This construction guarantees that we stay in $\mathcal{G}$. 
Double counting is used again to show that the proportion of bad graphs
goes to~$0$ as $n$ goes to infinity.

In the next proof we do not need all the power of this method,  since blocks
in $\ex(C_4)$  have bounded degree, but already in the class $\ex(C_5)$
there are blocks of arbitrary high degree.

\begin{lem}\label{lem:c4upper}
In the class $\ex(C_4)$ we have, for every $\epsilon > 0$,
$$
     \Delta_n \le (2+\epsilon) {\log n  \over \log\log n}.
$$
\end{lem}
\begin{proof}
We first observe that the only blocks in $\ex(C_4)$ are edges and triangles.
Let $\G=\ex(C_4)$ and let $\B_n \subseteq \G_n$ now denote the set of bad graphs with  
$$\Delta_n > (2+\epsilon) {\log n  \over \log\log n}.$$ 
Let $G$ be a graph in $\B_n$ and let $v$ be a vertex with degree $k >(2+\epsilon) {\log n  \over \log\log n}.$  Again, by Lemma~\ref{lem:isolated} and Lemma~\ref{lem:pendant1}, w.h.p. there are at least $(\alpha n - 2\log n/\log\log n)\geq 2 \alpha n/3$ pendant vertices not adjacent to $v$. Let $c=\min(\frac{\epsilon/2}{1+(\epsilon/2)}, \alpha/3)$. Let $r$ be the number of blocks incident to $v$ and observe that $(k/2) \le r \le k$, since the only blocks are edges and triangles.
Choose a set of $\lceil cr \rceil$ pendant vertices not adjacent to $v$ and delete their adjacent edges. Maintain vertex $v$ and delete all its adjacent edges. Attach the $\lceil cr \rceil \le 2\alpha n/3$ chosen vertices  to $v$, and attach the roots of all $r$ blocks in all possible ways to any of the previously added $\lceil cr \rceil$ vertices. The counting is as before:  the number of graphs constructed in this way is at least $\binom{2\alpha n/3}{cr}(cr)^r$, and for recovering $G$, one just has to reattach the $cr$ pendant vertices, giving rise to at most $n^{cr+1}$ choices. Hence,
$$
 {C(n) \over R(n) } \geq \frac{\binom{2 \alpha n/3}{cr}(cr)^r}{n^{cr+1}} \geq \frac{(\frac13 \alpha)^{cr}(cr)^r}{n (cr)! }.
$$
Note that $(cr)^r / (cr)! > (cr)^{(1-c)r}$. Thus, taking logarithms, this gives
$$
    \log {C(n) \over R(n) } \geq (1-c)r \log r - \log n - O(r) \geq (1-c)(k/2)\log k - \log n - O(k),
    $$
which again tends to infinity by our choice of $c$. Hence, $|\mathcal{B}_n|=o(|\mathcal{G}_n|)$.
\end{proof}

\begin{thm}\label{thm:upperc5}
In the class $\ex(C_5)$ we have, for every $\epsilon >0$,
$$
    \D_n \le (1+\epsilon) {\log n \over \log\log \log  n}     \qquad \hbox{w.h.p.}
$$
\end{thm}

\begin{proof}
Let  $\G = \ex(C_5)$ and let $\mathcal{B}_n \subseteq \mathcal{G}_n$ the graphs with $$\Delta_n > (1+\epsilon) \log n/\log\log\log n.$$

Assume for contradiction that there is some constant $\mu$ such that $|\mathcal{B}_n| \geq \mu |\mathcal{G}_n|$ infinitely often.
Let $G$ be a graph in $\mathcal{B}_n$ and let $v$ be a vertex of $G$ such that $k = \textrm{deg}(v) > (1+\epsilon) \log n/\log\log\log n$. As before, by Lemma~\ref{lem:isolated} and Lemma~\ref{lem:pendant1}, w.h.p. there are at least $(\alpha n - 2\log n/\log\log n)\geq 2 \alpha n/3$ pendant vertices not adjacent to $v$.
The strategy of the proof is as follows. We partition the blocks incident with $v$ according to their type and to their contribution to the degree of $v$. Those with degree smaller than a threshold can be safely ignored for the asymptotics.
Those of large degree, which by Lemma~\ref{lem:c5} are isomorphic to either $K_{2,t}$ or $K_{2,t}^+$,  are used to produce many new graphs as in the proofs for the lower bounds.
Then a double counting argument is used again to show that  $|\mathcal{B}_n| / |\mathcal{G}_n| \to 0$.
The strategy for $\ex(C_6)$ and $\ex(C_7)$ is very similar but there are more types of blocks to consider, making the situation a bit cumbersome.

Let us proceed with the proof. We partition the $2$-connected components (blocks) attached to~$v$. Using Lemma~\ref{lem:c5}, they can be partitioned  into the following classes:
\begin{enumerate}
\item Blocks contributing to $\textrm{deg}(v)$ at most $\frac{\log k}{\log \log k}$. That is, these are blocks whose root degree is at most $\frac{\log k}{\log \log k}$.
\item Blocks of type $K_{2,t}$ with $t > \frac{\log k}{\log \log k}$
\item Blocks of type $K_{2,t'}^+$ with $t' > \frac{\log k}{\log \log k}$
\end{enumerate}
Let $r_i$ be the number of blocks of class $i$ and denote by $k_i$ the total contribution of edges belonging to a block of class $i$ to $\textrm{deg}(v)$. Clearly, $k=k_1+k_2+k_3$, and also observe that $r_1 \geq \frac{k_1 \log \log k}{\log  k}$ and that $r_i < \frac{k_i \log \log k}{\log k}$ for $i=2,3$.

In order not to run out of pendant vertices, let now $c=\min(\frac{\epsilon/2}{1+\epsilon}, \frac13\alpha)$. From $G$ we construct now a class of  graphs, as follows.

\begin{itemize}
\item Choose a set of $h$ ($h$ will be determined below) pendant vertices $U$ not incident to $v$ and delete their adjacent edges.
Maintain vertex $v$ and delete all its adjacent edges.
Choose three vertices from $U$, eliminate them from $U$ and make them neighbors of $v$. Call them w.l.o.g. $v_1,v_2,v_3$ and assume that their labels are sorted increasingly.
Choose $\lceil cr_1 \rceil$ vertices from $U$, eliminate them from $U$ and make them neighbors of $v_1$.
Attach the roots of all blocks of class $1$ in all possible ways to any of the previously added $\lceil cr_1 \rceil$ vertices.
Choose $r_2$ vertices from $U$, eliminate them from $U$ (each of them representing a block of class $2$) and make them neighbors of $v_2$.

\item  For each block of class $2$ of type $K_{2,t_i}$ ($i=1,\ldots,r_2$) choose $1+\lceil ct_i \rceil$  vertices from $U$, eliminate them from $U$, and connect all of them to the previously added vertex that represents the $i$-th block of this class. Let $x_i$ be the vertex with smallest label among the $1+\lceil ct_i \rceil$ vertices added ($i=1,\ldots,r_2$).
For each block $K_{2,t_i}$ of $G$, define $z_i^0$ to be the other vertex apart from $v$ of degree $t_i$, and let $z_i^1,\ldots,z_i^{t_i}$ the vertices of degree $2$. In our construction, we delete all edges belonging to the original block and we add the following edges: $z_i^0$ is connected with $x_i$, and we connect each of the vertices $z_i^j$ ($j \geq 1$) in all possible ways to any of the previously added $\lceil ct_i \rceil$ vertices excluding $x_i$.

\item
For blocks of class $3$, do the analogous steps as for blocks of type $2$.
\end{itemize}

\noindent
Observe that the new vertices have been added in a tree-like way in this construction, that is, we have not created any cycle that did not exist in the original graph. In particular, if $G \in Ex(H_1,\dots,H_k)$, so are all the newly constructed graphs.
Also observe that the number $h$ of pendant vertices used is at most $ck (1+o(1)) < \alpha n$.

We proceed to count the number of different graphs we obtain by applying this construction to one graph of $\mathcal{B}_n$. To simplify notation, we will ignore ceilings.
We obtain at least
\begin{align} \label{eq:constructions}
 & \binom{2 \alpha n/3}{h}\binom{h}{cr_1,r_2,ct_1+1,\ldots,ct_{r_2}+1,r_3,ct'_1+1,\ldots,ct'_{r_3}+1,3} r_2! r_3! \ \times  \nonumber \\
 & (cr_1)^{r_1} \left(\prod_{i=1}^{r_2} (ct_i)^{t_i} \right) \left( \prod_{i=1}^{r_3} (ct'_i)^{t'_i}\right)
\end{align}
many graphs, since there are at least $\binom{2\alpha n/3}{h}$ ways to choose $h$ pendant vertices not incident to $v$, which then have to be partitioned into the different groups explained before (yielding the multinomial coefficient). The factors $r_2!$ and $r_3!$ come from the fact that blocks of class $2$ and $3$ are distinguishable because of their labels, hence any permutation of the $r_2$ and $r_3$ vertices will give rise to different graphs. The last group of three vertices in the multinomial coefficient corresponds to the vertices $v_1,v_2,v_3$ (there is no $3!$, since the roles of these vertices are determined by their labels). The remaining factors count the possible ways to do the connections between the $t_i $ vertices and the added $ct_i$ vertices, and between the $t_i'$ and the $ct_i'$.

Since different original graphs may give rise to the same new graph, we have to divide the total number of constructions by the number of preimages of a new graph. This number is as before at most $n \cdot n^h$, since we first must guess the vertex $v$ of the original graph (this gives the factor $n$) and then we have to redistribute the $h$ newly added vertices as pendant vertices (for those we have at most $n^h$ choices).

Our goal is to show that the total number of newly constructed graphs divided by the number of preimages of a new graph tends to infinity as $n$ increases, hence contradicting the assumption that $|\mathcal{B}_n| \geq \mu |\mathcal{G}_n|$ for infinitely many values of $n$.

Note that the following expression is a lower bound of (\ref{eq:constructions}).
\[ \left(\frac{1}{3}(\alpha-c)n\right)^h (cr_1)^{(1-c)r_1}\prod_{i=1}^{r_2}(ct_i)^{(1-c)t_i} \prod_{i=1}^{r_3}(ct'_i)^{(1-c)t'_i}, \]
where we have used that $h=ck(1+o(1))$, $k<n$ so that $\frac16(\frac23 \alpha n)!/(\frac23\alpha n -h)!$ is bounded from below by $(\frac{1}{3}(\alpha-c)n)^h$; we also used that for any $g>0$ it holds that $(cg)^g/(cg)! \ge (cg)^{(1-c)g}$, and that for any $g$ such that $cg\ge 3$ it holds that $(cg)^g/(cg+1)!\ge (cg)^{(1-c)g}$.

We now divide by the number of preimages $n\cdot n^h$, and then we take logarithms. Hence, noting that $k_2=\sum_{i=1}^{r_2} t_i$ and $k_3=\sum_{i=1}^{r_3} t'_i$, we obtain
\[ -\log n -O(h)+(1-c)r_1\log r_1 -O(r_1) +(1-c)\sum_{i=1}^{r_2}t_i \log t_i -O(k_2) +(1-c)\sum_{i=1}^{r_3}t'_i \log t'_i -O(k_3). \]
By Lemma~\ref{lem:entropy}, $\sum_{i=1}^{r_2}t_i \log t_i$ is minimal when all $t_i$, and the same applies to the $t'_i$. Hence, the previous expression is bounded from below by
\begin{equation}\label{eq:final}
-\log n -O(k) +(1-c)\left( r_1\log r_1 + k_2\log \frac{k_2}{r_2} + k_3\log \frac{k_3}{r_3} \right).
\end{equation}

 Now, letting $k_i = \beta_i k$ for $i=1,2,3$, we obtain
\[ r_1 \geq \frac{k_1 \log \log k}{\log k} \geq \beta_1 \frac{k \log \log k}{\log k},\]
and thus
\[ r_1\log r_1 \ge \beta_1 \frac{k \log \log k}{\log k}(\log k+o(\log k)) = \beta_1 k \log \log k (1+o(1)).\]
Also, recall that $r_2\le \frac{k_2 \log \log k}{\log k}$, so that
\[ \frac{k_2}{r_2} \ge \frac{\log k}{\log \log k}, \]
and the term $k_2 \log \frac{k_2}{r_2}$ in (\ref{eq:final}) is at least
\[ k_2 \log \frac{k_2}{r_2} \ge k_2 \log \log k(1-o(1)) = \beta_2 k \log\log k (1-o(1)).\]
By the same argument, $k_3 \log \frac{k_3}{r_3} \ge   \beta_3 k \log\log k (1-o(1)).$ As $\beta_1+\beta_2+\beta_3=1$, one of the $\beta_i$ has to be at least $\frac13$,  hence we can safely ignore the term $-O(k)$ in~\eqref{eq:final}. The expression in~\eqref{eq:final} is thus bounded from below by
$$
(1-o(1))(1-c)k \log\log k - \log n \geq (1-o(1))(1+\epsilon/2)\log n - \log n,
$$
which tends to infinity, as desired.
\end{proof}


\begin{thm}\label{thm:upperc6}
In the class $\ex(C_6)$ we have, for a suitable constant $C>0$,
$$
    \D_n \le C {\log n \over \log\log\log n}     \qquad \hbox{w.h.p.}
$$
\end{thm}

\begin{proof}
The proof starts as for $\ex(C_5)$. Let  $\G = \ex(C_6)$ and let $\mathcal{B}_ n \subseteq \mathcal{G}_n$ the class of graphs with $$\Delta_n > C \log n/\log\log\log n.$$
We assume for contradiction that there is some constant $\mu$ such that $|\mathcal{B}_n| \geq \mu |\mathcal{G}_n|$ infinitely often.
Let $G$ be in $\B_n$ and let $v$ be a vertex of $G$ such that
$$k=\deg(v) \geq \frac{C \log n}{\log \log \log n}
$$
for some constant $C$ large enough. By Lemma~\ref{lem:isolated} and Lemma~\ref{lem:pendant1}, w.h.p. there are at least $2 \alpha n/3$ pendant vertices not incident to $v$.
Let $c=\min(1-\frac{1+\epsilon/2}{C},\frac13 \alpha)$. We partition the $2$-connected components (blocks) attached to $v$ into different classes (see Lemma~\ref{lem:c6}).

\begin{enumerate}
\item Blocks contributing to $\deg(v)$ at most $\frac{\log k}{\log \log k}$.
\item Blocks of type $K_{2,s}$ and $K_{2,s}^+$ with $s > \frac{\log k}{\log \log k}$.
\item Blocks of type $H_{2,s,t}$ or $H_{2,s,t}^*$.
\item Blocks of type $\widetilde{H}_{2,s,t}$ or $\widetilde{H}_{2,s,t}^*$ (see the remark after Lemma~\ref{lem:c6}).
\end{enumerate}
Choose a set of $h$ pendant vertices $U$ not incident to $v$ and delete their adjacent edges. Maintain vertex $v$ and delete all its adjacent edges. We now have a bounded number $N$ of subclasses  represented by classes 1 to 4 and the possible cases in the definition of ${H}^*_{2,s,t},\widetilde{H}_{2,s,t}$ and $\widetilde{H}^*_{2,s,t}$. For each subclass $i$, let $r_i$ be the number of blocks of subclass $i$ incident with $v$. For each $i$, take a pendant vertex $w_i$ from $U$ and make it adjacent to $v$, and sort the $w_i$ in increasing order of the labels. For each $i$ (except for class 1), take $r_i$ pendant vertices from $U$ and make them adjacent to $w_i$.

For blocks in classes $1$ and $2$ (they give rise to $r_1$,$r_2$, $r_3$), the $r_i$ play the same role as in
the proof of Theorem \ref{thm:upperc5}, and we append the same construction as there.

For blocks of type $H_{2,s,t}$ the construction is very similar; they behave like the graphs $K_{2,s}$, but with two sets, of size $s$ and $t$, of vertices of degree two.
For each block of type $H_{2,s,t}$, we add two new sorted vertices from $U$ and make them adjacent to the vertex representing the block.
Take $2+ cs $ and $2+ ct $ vertices from $U$ (ignoring ceilings from now on) and connect  them, respectively,  to the two previously added vertices.
Let $x_0, x_1$ and $y_0, y_1$, respectively,  be the vertices with smallest labels (in this order) among the $2+  cs$ and the $2+ct$ added vertices.
Delete all edges belonging to the original block and the $s$ vertices to the newly added  $cs$ vertices (excluding $x_0$ and $x_1$) in all possible ways, and do the same for the $t$ vertices (excluding $y_0$ and $y_1$). Also, connect $x_0$ to $v_1$ (notation as in Lemma~\ref{lem:c6}), $x_1$ to $v_2$, and $y_0$ to $v_3$, $y_1$ to $v_4$.
For blocks of type $H_{2,s,t}^*$ the construction is exactly the same;
the fact that the different subclasses are identified by the labels as well as the special role of $v_1,v_2,v_3,v_4$ guarantees unique reconstruction.
Finally, consider blocks of type $\widetilde{H}_{2,s,t}$, and assume without loss of generality that $v_2$ plays the role of $v$. In this case we add only $cs+4$ vertices
from $U$ and make them adjacent to the vertex representing the block.
Let $x_0$, $x_1$, $x_2$ and $x_3$ be the vertices with the four smallest labels (in this order). Delete all edges in $\widetilde{H}_{2,s,t}$ emanating from $v$ and $v_2$, connect all the $s$ neighbors of $v_2$ (excluding $v_3$)  to the $cs$ vertices (excluding $x_0$,$x_1$,$x_2$ and $x_3$) in all possible ways. Connect $x_0$ to $v$, $x_1$ to $v_3$, $x_2$ to $v_1$, and $x_3$ to $v_2$. As before, the same construction is applied for $\widetilde{H}_{2,s,t}^*$ (the only difference being that all optional edges are deleted as well); as before, the special roles and the different labels of different subclasses provide all information for unique reconstruction.

\medskip
Since  no new cycle is created, given $v$ and the new graph, we can uniquely determine the original graph it comes from. Observe also that we used only $h \leq ck(1+o(1))$ pendant vertices. As before, we count the number of different graphs we obtain by applying this construction, yielding similar multinomial coefficients and other factors. Dividing by the number of preimages of a new graph, which is at most $n^{h+1}$, and taking logarithms, we obtain
$$
\renewcommand{\arraystretch}{1.5}
\begin{array}{l}
-\log n-O(h)+(1-c)r_1 \log r_1-O(r_1)+(1-c)\sum_{j=1}^{r_2}(s_j)_2 \log (s_j)_2-O(k_2)+
\\(1-c)\sum_{j=1}^{r_3}(s_j)_3 \log (s_j)_3
-O(k_3)+(1-c)\sum_{i\geq 4} \sum_{j=1}^{r_i}(s_j)_i \log (s_j)_i+\\(1-c)\sum_{i\geq 4} \sum_{j=1}^{r_i}(t_j)_i \log (t_j)_i - O(\sum_{i\geq 4} k_i),
\end{array}
$$
where we denote by $k_i$ the total contribution of blocks of subclass $i$ to the degree of $v$, and by $(s_j)_i$ and $(t_j)_i$ the corresponding sizes of the $j$th block of subclass $i \geq 4$ (note that the second sum over $i \geq 4$ does not apply to subclasses in $\widetilde{H}_{2,s,t}^*$).
By Lemma~\ref{lem:entropy} (applied twice, to each block and then to each subclass), the previous expression is at least

\begin{equation}\label{eq:c6}
-\log n-O(k)+(1-c)(r_1 \log r_1+k_2 \log \frac{k_2}{r_2}+k_3 \log \frac{k_3}{r_3}+\sum_i k_i \log \frac{k_i}{r_i}).
\end{equation}

Since $\sum_i k_i=k$, there exists some $1 \leq i \leq N$ such that $k_i \geq k/N$. If this is true for $i=1$, then
$$r_1 \log r_1 \geq \frac{k_i \log \log k}{\log k}(\log k_i+o(\log k_i))=\frac{k \log \log k}{N}(1+o(1)).
 $$
 	Otherwise, if $i \geq 2$, since $\frac{k_i}{r_i} \geq \frac{\log k}{N \log \log k}$, as before, $k_i \log \frac{k_i}{r_i} \geq  \frac{k \log \log k}{N}(1+o(1))$. Thus, by our choices of $C$ and $c$, \eqref{eq:c6} tends to infinity as desired.
\end{proof}


In the class $\ex(C_7)$,  the right order of magnitude of the expected maximum degree changes, compared to $\ex(C_5)$ and $\ex(C_6)$. Before going into the proof, we give some intuition about the different behaviour in $\ex(C_7)$. The existence of the component $V_{s,t,E}$ as described in Lemma~\ref{lem:c7}, and in particular the existence of $t$ stars of different degrees $q_i$ inside one block, gives rise to new constructions. In order to ensure many constructions, at both levels choices have to be made: if there were few stars of a high degree, only on the second level many choices can be made, but if, however, there are many stars of small degree, on the first level many choices can be made. For a medium number of stars with medium degree, on both levels some choices can be made.
 These two choices imply that the definition of \emph{small} has to be changed, and the trade-off between the contribution of small blocks and other larger blocks (which give different types of contributions in the proofs) gives rise to an additional application of the logarithm. We now state the result for this class.

\begin{thm}\label{thm:upperc7}
 In the class $\ex(C_7)$ we have, for a suitable constant $C>0$,
$$
    \D_n \le C {\log n \over \log\log\log\log n}     \qquad \hbox{w.h.p.}
$$
\end{thm}

\begin{proof}

Let $\mathcal{G}=\ex(C_7)$.
The proof starts as for $\ex(C_5)$ and $\ex(C_6)$. Let $\mathcal{B}_n \subseteq \mathcal{G}_n$ the graphs with $$\Delta_n > C \log n/\log\log\log \log n.$$
We assume once more for contradiction that there is some constant $\mu$ such that $|\mathcal{B}_n| \geq \mu |\mathcal{G}_n|$ infinitely often. Let $G$ be a graph in $\mathcal{B}_n$ and
let $v$ be a vertex of $G$ such that
$$k=\deg(v) \geq \frac{C \log n}{\log \log \log \log  n}
$$
for some constant $C$ large enough. Again, by Lemma~\ref{lem:isolated} and Lemma~\ref{lem:pendant1}, w.h.p. $G$ has at least $2\alpha n/3$ pendant vertices not incident to $v$. Let $c=\min(1-\frac{1+\epsilon/2}{C},\frac13 \alpha)$.
As before, we partition the $2$-connected components (blocks) attached to $v$ into different classes. Using Lemma~\ref{lem:c7}, whose notation is used in the following (see also the remark following Lemma~\ref{lem:c7}), we may partition them into
\begin{enumerate}
\item Blocks contributing to $\deg(v)$ at most $\frac{\log k}{\log \log \log k}$.
\item Blocks of type $K_{2,s}$, $K_{2,s}^+$,  $H_{2,s,t}$, $H_{2,s,t}^*$, $\widetilde{H}_{2,s,t}$ and $\widetilde{H}_{2,s,t}^*$
\item Blocks of type $S_{s,t,u,w}$, $V_{s,t,E}$, and the corresponding graphs $S_{s,t,u,w}^*$, $V_{s,t,E}^*$
\item Blocks of type $\widetilde{S}_{s,t,u,w}$, $\widetilde{V}_{s,t,E}$, and the corresponding graphs $\widetilde{S}_{s,t,u,w}^*$, $\widetilde{V}_{s,t,E}^*$
\end{enumerate}
Choose a set of $h$ pendant vertices $U$ not incident to $v$ and delete their adjacent edges. Maintain vertex $v$ and delete all its adjacent edges.
We still have a bounded number of subclasses  represented by the different classes and the possible optional edges. For each subclass $i$, let $r_i$ be the number of blocks of subclass $i$ incident with $v$. For each $i$, take a pendant vertex $w_i$ from $U$ and make it adjacent to $v$, and sort the $w_i$ in increasing order of the labels. For each $i$ (except for class 1), take $r_i$ pendant vertices from $U$ and make them adjacent to $w_i$.

For blocks in classes 1 and 2, we proceed as in the proof of Theorem~\ref{thm:upperc5} and Theorem~\ref{thm:upperc6}. We ignore ceilings and justify after the constructions that they may be safely disregarded. For blocks $S_{s,t,u,w}$ and $S_{s,t,u,w}^*$ the construction is very similar as before: for the new vertex $b$ (among the $r_i$ added ones) representing a block of such a subclass, take three   sorted vertices from $U$ and make them adjacent to $b$. Take $5+cs$, ($ct$, $cu$, respectively) vertices from $U$, and add them to the first of these sorted vertices (second and third, respectively). Denote by $x_1$,$x_2$,$x_3$,$x_4$,$x_5$ the vertices with smallest labels (in this order) of the first group. Delete all edges from the original block except for the edges incident to the $w$ vertices (excluding $v$,$v_1$,$v_3$,$v_5$, if the edges are present) that are connected with both $v_2$ and $v_4$. Append the special vertices $v_1$, $v_2$, $v_3$, $v_4$ and $v_5$ to the vertices $x_1$,$x_2$, $x_3$, $x_4$, $x_5$ in this order. Connect then the $s$ vertices (which originally were adjacent to $v$ and $v_2$) to the $cs$ vertices of the first group (excluding $x_1,\ldots,x_5$) in all possible ways, and do the analogous construction for the $t$ and $u$ vertices. Note that this time we might construct cycles of length $6$ (of the type $bx_2v_2av_4x_4b$), where $a$ is one of the $w$  vertices connecting $v_2$ and $v_4$, but by the special roles of special vertices unique reconstruction is still guaranteed.

For blocks of type $V_{s,t,E}$ and $V_{s,t,E}^*$,
and its corresponding vertex $b$ representing the block, take eight sorted vertices
$b_1,\ldots,b_8$ from $U$ and make them adjacent to $b$. Take $ce_1,ce_2,ce_3,ce_4$ elements from $U$ and add them to $b_1,b_2,b_3,b_4$, respectively. Take $5+cs$ elements from $U$ (call the vertices with the $5$ smallest labels $x_1,\ldots,x_5$, in this order, as before), make them adjacent to $b_7$, and take $ct$ elements from $U$ and make them adjacent to $b_8$. From the original block delete all edges emanating from $v, v_2, v_4$, all edges between special vertices, all edges going between $v_3$ and any of the $e_1$, $e_2$ vertices of the first and second group of $E$.  For the $e_3$ graphs of the third group of $E$, for any $1 \leq i \leq e_3$ the edges between the vertices of degree $q_i$ (different from $v_3$) and its $q_i$ neighbors of degree $2$ are retained, and all others are deleted, and analogously for the $e_4$ graphs of the forth group. For the $e_5$ and $e_6$ graphs of the fifth and sixth group of $E$, all edges of it are deleted if the vertex of degree $q_i$ (different from $v$) satisfies $q_i > \frac{\log \log \log n}{\log \log \log \log n}$, otherwise all edges going between the vertex of degree $q_i$ (different from $v$) and its $q_i$ neighbors different from $v$ and $v_3$ are retained and the others are deleted. Now, connect $v_1,\ldots,v_5$ with $x_1,\ldots,x_5$. For the $e_1$ vertices originally connecting $v$ and $v_3$, connect them to the $ce_1$ vertices (which were attached to $b_1$) in all possible ways. For the $e_2$ pairs adjacent to each other and both connecting $v$ and $v_3$, connect the one with smaller label in all possible ways to the $ce_2$ vertices attached to $b_2$ (recall that the edge connecting such a pair is not deleted). For the $e_3$ and $e_4$ double stars $K_{2,q_i}$, connect all vertices of degree $q_i$ (different from $v_3$) and its pending $q_i$ neighbors with the $ce_3$ and $ce_4$ vertices attached to $b_3$ and $b_4$, respectively, in all possible ways.  For the $e_5$ graphs $K_{2,q_i}$ emanating from $v$ (of degrees $q_1,\ldots, q_{e_5}$), take $\frac{c}{2}e_5$ vertices from $U$, attach them to $b_5$, and connect each of the $e_5$ vertices $z_1,\ldots,z_{e_5}$ of degree $q_1,\ldots,q_{e_5}$ to the $\frac{c}{2}e_5$ vertices in all possible ways. Then, for each of the $z_i$ ($1 \leq i \leq e_5$), do the following:  if $q_i \leq \frac{\log \log \log n}{\log \log \log \log n}$, do nothing (recall that the neighbors of $z_i$ are still pending). Otherwise, take $\frac{c}{2}q_i$ vertices from $U$ and make them adjacent to $z_i$. Connect each of the $q_i$ vertices (originally neighbors of $z_i$) in all possible ways to the newly attached $\frac{c}{2}q_i$ vertices. The analogous construction is done for $e_6$ (with $b_6$ instead of $b_5$).  Finally,  connect the $s$ vertices originally connecting $v$ and $v_2$ (excluding special vertices) with the the group of $cs$ new vertices (excluding $x_1,\ldots,x_5$) attached to $b_7$ in all possible ways.  Similarly, connect the $t$ vertices originally connecting $v$ and $v_4$ with the group of $ct$ new vertices attached to $b_8$ in all possible ways. Here, the graph constructed is always a tree, and reconstruction is unique.

For blocks of type  $\widetilde{S}_{s,t,u,w}$ and $\widetilde{S}_{s,t,u,w}^*$, the strategy is similar as before. Assume without loss of generality that $v_2$ plays the role of $v$. In this case we take three vertices from $U$ (sorted) and make them adjacent to the vertex representing this block. Take $5+cs$ new vertices from $U$, make them adjacent to the first one, then $cu$ further ones, make them adjacent to the second one, and finally another $cw$, which are made adjacent to the third one. All edges are deleted except for edges between $v_4$ and its $t$ non-special neighbors that were also connected with $v$. The $5$ vertices of the first group with smallest labels are connected to special vertices, and the $s$,$u$ and $w$ neighbors of $v_2$ (except for special vertices) are, as before, connected in all possible ways with the $cs$, $cu$ and $cw$ vertices of the respective groups. Observe that the constructed graph is a tree.

For blocks of type  $\widetilde{V}_{s,t,E}$ and  $\widetilde{V}_{s,t,E}^*$, either of $v_2$, $v_3$, $v_4$ or any of the external vertices in double stars of degree $q \geq \frac{\log k}{\log \log \log k}$ arising in the groups $e_3,e_4,e_5,e_6$ may play the role of $v$. In all cases, edges between special vertices are always deleted. If $v_3$ plays the role of $v$, all edges between $v_2$ and its $s$ neighbors that are connected with $v$, and all edges between $v_4$ and its $t$ neighbors that are connected with $v$ are retained. The others are deleted, and the same construction is applied as for $V_{s,t,E}$ and $V_{s,t,E}^*$, with $v_3$ playing the role of $v$. If $v_2$ or $v_4$ (assume $v_2$ without loss of generality) plays the role of $v$, all edges emanating from a neighbor of $v_2$ are deleted. In order not to create a cycle of length $7$, for any of the $e_2$ adjacent pairs between  $v$ and $v_3$, both edges connecting $v_3$ with either of them are deleted. For the $e_3$ and $e_4$ graphs of the third and forth group of $E$, the $e_3$ and $e_4$ edges emanating from $v$ to these vertices are deleted, and for the $e_5$ and $e_6$ graphs of the fifth and sixth group of $E$, the $e_5$ and $e_6$ edges emanating from $v_3$ to these vertices are deleted. For the $s$ edges emanating from $v_2$ the usual reconstruction is performed (again with $5$ special vertices assuring unique reconstruction). If any of the external vertices $a$ of degree $q$ plays the role of $v$, the procedure is very similar: deletion of edges of groups $e_1,\ldots,e_6$ (except for those going to $a$ and to neighbors of $a$ different from both $v$ and $v_3$, which are all deleted) is as in the previous case. The edges emanating from the $s$ and $t$ vertices that are connected to $v$ and $v_2$, and to $v$ and $v_4$, are retained. Then, as usual, $5+cq$ vertices are taken from $U$, and the $q$ neighbors of $a$ are connected in all possible ways to the $cq$ new vertices (the $5$ vertices take care of special vertices). Note that again cycles of length $6$ can be constructed, but by the special roles of special vertices reconstruction is still unique (for example, in the reconstruction, $v_3$ is connected to all neighbors of $v$, except for the vertex it was attached to, and except for degree $2$ vertices in a block of size at least $4$ attached to $v$. In this way, $v_3$ will be connected to both vertices of type $e_2$, but not to external vertices of degree $2$ of some $K_{2,q_i}$.)

Observe that the largest cycle created is of length at most $6$, and in all cases the special vertices guarantee unique reconstruction. Observe also that the number of pendant vertices used is at most $h = ck(1+o(1))$: for contributions of type $e_5$ and $e_6$ in components $V_{s,t,E}$,$V_{s,t,E}^*$ (and of type $e_3$ and $e_4$ in components $\widetilde{V}_{s,t,E}$,$\widetilde{V}_{s,t,E}^*$ with $v_3$ playing the role of $v$), at the first level $\frac{c}{2}e_5$ vertices are used, and at the second level, at most $\frac{c}{2}\sum_{i=1}^{e_5} q_i(1+o(1))$ (note that ceilings may be safely disregarded, as only for sufficiently large $q_i$ these vertices are chosen), and since $e_5 \leq \sum_{i=1}^{e_5} q_i$, the total number is at most $c \sum_{i=1}^{e_5} q_i$. For the other contributions it is obvious. As before, for each case we count the number of different graphs we obtain by applying this construction, yielding similar multinomial coefficients and other factors as before. Then we divide by the number of preimages of a new graph, which is at most $n^{h+1}$, and take logarithms.
Similar calculations as before show that the most negative term is $-\log n$, coming from the choice of the vertex $v$. Let $N$ be the total number of types of subclasses (recall that it is still constant).

Now, if at least $k/N$ of the degree of $v$ is in blocks of size $\frac{\log k}{\log \log \log k}$, then the number of such blocks $r_1$ is at least $\frac{k \log \log \log k}{N \log k}$. By the same arguments as before, the constructions of these blocks give a term $r_1 \log r_1 \geq \frac{ k \log \log \log k}{N \log k}(\log k + o(\log k))=\frac{k}{N} \log \log \log k(1+o(1))$, and for $C$ large enough this is bigger than the (negative) term $\log n$.
Otherwise, suppose that at least $k/N$ of the degree of $v$ results from any fixed class of blocks  excluding $V_{s,t,E}$ or $V_{s,t,E}^*$ (and also excluding $\widetilde{V}_{s,t,E}$ and $\widetilde{V}_{s,t,E}^*$ with $v_3$ playing the role of $v$). Letting $r_j$ denote the number of such blocks, by similar calculations as before, as there is only one level of choice, we obtain a positive term $\Theta( k \log \frac{k}{r_j})$. Since $r_j \leq \frac{k \log \log \log k}{\log k}$,
$$
\Theta( k  \log \frac{ k}{r_j}) = \Omega( k \log \log k),$$ which is asymptotically bigger than $\log n$.

Hence, assume that $k/N$ of the degree of $v$ comes from the subclass in $V_{s,t,E}$ or $V_{s,t,E}^*$ (or $\widetilde{V}_{s,t,E}$ and $\widetilde{V}_{s,t,E}^*$ with $v_3$ playing the role of $v $), and assume without loss of generality that it is  the class $V_{s,t,E}$. Let again be $r_j \leq \frac{k \log \log \log k}{\log k}$ the number of blocks of this class. If at least $k/(2N)$ of the total degree comes from contributions of the groups of $s$, $t$, $e_1,e_2,e_3,e_4$ in the blocks of $V_{s,t,E}$, then, as before, only considering those terms, as there is one level of choice, we obtain a term $\Theta( k \log \frac{k}{r_j}) \gg \log n.$

Hence, we may assume without loss of generality that  $k/(4N)$ of the total degree comes from contributions of group $e_5$.  Once more, we split this into two subcases: if at least $k/(8N)$ of the total degree comes from double stars $K_{2,q}$ with  $q \leq \frac{\log \log \log n}{\log \log \log \log n}$, then at least $z\geq \frac{k \log \log \log \log n}{8N \log \log \log n}$ such double stars $K_{2,q}$ are needed. Denote by $z_i$ the number of double stars inside the $i$th block to $z$, for $1 \leq i \leq r_j$. Each such block gives a term $z_i \log z_i$, and the total contribution is minimized when the number of double stars is equally split among all blocks. Assuming the worst case of $r_j =\frac{k \log \log \log k}{\log k}$ and $z= \frac{k \log \log \log \log n}{8N \log \log \log n}$, the total contribution is thus at least
$$
(1-c)\left(z \log \frac{z}{r_j}\right) =(1-c)\left( z \:(1+o(1))\log \log \log n \right) = (1-c)\frac{C \log n}{8N}(1+o(1)),
$$
which for $C$ large enough is bigger than $\log n.$  If on the other hand at least $k/(8N)$ of the total degree comes from double stars $K_{2,q}$ with $q > \frac{\log \log \log n}{\log \log \log \log n}$, then first observe that the number $z$ of double stars $K_{2,q}$ contributing to the total degree satisfies
$z \leq \frac{k \log \log \log \log n}{8N\log \log \log n}.$ Denote again by $q_i$ the degree of the $i$th double star, for $1 \leq i \leq z.$  Clearly, $\sum_{i=1}^z q_i \geq k/(8N)$. Each such double star gives on the second level of choice rise to a term $(1-c)q_i \log q_i$. Assume again the worst case $\sum_{i=1}^z q_i = k/(8N)$ and $z=\frac{k \log \log \log \log n}{8N\log \log \log n}.$ This contribution is clearly minimized if the contribution is split evenly, that is, $q_i =\frac{k}{8Nz}$, and in this case we obtain
$$
(1-c)\left( \frac{k}{8N} \log \frac{k}{8Nz}\right) = (1-c) \left( \frac{k}{8N}(1+o(1)) \log \log \log \log n \right)=(1-c)\frac{C \log n}{8N}(1+o(1)).
$$
which for $C$ large enough again is bigger than $\log n.$ Hence, in all cases, $C(n)/R(n) \rightarrow \infty$, as desired, and the proof is finished.

\end{proof}

We conclude this section with a combinatorial proof of a result previously obtained by analytic methods.

\begin{thm}
In the class of outerplanar graphs we have the following upper bound:
$$
    \D_n \le c \log n \qquad \hbox{w.h.p.},
$$
where $c > 0$ is a suitable constant.
\end{thm}

\begin{proof}
Let $\mathcal{G}$ be the class of outerplanar graphs and let $\mathcal{B}_n \subseteq \mathcal{G}_n$ be the graphs with $\Delta_n > c \log n$, where $c$ is a sufficiently large constant. Let $d=\lfloor  c \log n\rfloor$. As usual, we assume for contradiction that $|\mathcal{B}_n| \geq \mu |\mathcal{G}_n|$ for some $\mu > 0$.

Let $G$ be a graph in $\B_n$ and let $v$ be a vertex of degree $d$. Let $k$ be the number of blocks containing $v$.
Once more, w.h.p., in $G$ there are at least $2 \alpha n/3$ pendant vertices not incident to $v$. As before, we want to show that $C(n) / R(n) \rightarrow \infty$, yielding the desired contradiction.
The construction procedure depends on whether $k$ is smaller or larger than $\beta \log n$, where $\beta>0$ is fixed constant (the value of $\beta$ is irrelevant, we could  take for instance  $\beta =1$).
We also fix positive constants $\beta' < \beta$ and $\beta''> \beta$.

\textit{Case 1: $k \ge \beta\log n$.}
 We detach the $k$ blocks from $v$ and form $k$ graphs with a pointed vertex each, the vertex corresponding to $v$. We choose $\beta'\log n$ pendant vertices $v_1,\dots, v_{\beta' \log n}$ that originally were not neighbors of $v$, delete the edges incident with them and make them neighbors of $v$. For every detached block we join its pointed vertex to one of the $v_i$, arbitrarily. The number of constructions is at least
$$
    {2\alpha n/3\choose \beta'\log n}(\beta'\log n)^{\beta\log n}.
$$
In order to recover a graph constructed in this way, we only need to guess $v$ and reattach the pendant vertices
to the original graph. This can be done in at most $n \cdot n^{\beta'\log n}$ ways. Using the previous notation, we have that
$$
    {C(n) \over R(n)} \ge {  \ds{2\alpha n/3\choose \beta'\log n}(\beta'\log n)^{\beta\log n}
                            \over n \cdot n^{\beta'\log n} } \geq (\beta' \log n)^{(\beta-\beta')\log n} \left(\frac{(2\alpha/3)^{\beta' \log n}}{n}\right),
                            $$
where we have used that $\binom{n}{k} \geq (n/k)^k$. As desired, this quantity tends to infinity as $n \to \infty$.

\textit{Case 2: $k < \beta\log n$.}
Fix an ordering of the $k$ blocks and draw each of them with all the vertices in the outer face.
This gives an ordering on the neighbors $u_1,\dots,u_d$ of $v$. Delete all the edges  incident with $v$.
Select an ordered list of $\beta''\log n$ pendant vertices, which were not originally neighbors of $v$, and make a path between them, according to the order. Call these vertices (in this order)  $v_1,\dots, v_{\beta'' \log n}$.
We join the $u_i$ to the $v_j$ while maintaining the outerplanar embedding, and with the restriction that no two $u_i$ coming from different blocks are joined to the same $v_j$.

For the $d$ neighbors $u_1,\ldots,u_d$ of $v$ we have at least $(\beta''-\beta)\log n$ choices where to make the decision to connect to the next $v_j$ (at most $\beta \log n$ times we are forced to switch to the next $v_j$).  The number of constructions is therefore at least
$$
    (2\alpha n/3)_{\beta''\log n} {d \choose \beta''\log n- \beta\log n}
    \ge \left({\alpha n  \over 2} \right)^{\beta''\log n} \left({c \log n^{} \over (\beta''-\beta)\log n}\right)^{(\beta''-\beta)\log n}
$$
In order to recover a graph constructed in this way we have to guess the vertex $v_1$. The path is identified by finding the cut edges: whenever at some vertex $v_j$ the vertices $u_i$ of a new block connect to it, the preceding edge is a cut edge. As all other edges emanating from a vertex of the path are not cut edges, these can be identified easily, and t each such cut edge the right orientation on the boundary of the next block, and thus of the vertices of this block, has to be found (if the block is an isolated vertex or a bridge then the edge emanating from the path to the block is also a cut edge, so there is also a factor $2$ in this case). In addition, the $v_j$ must be reattached to the original graph. This gives at most
$$
    n\cdot n^{\beta''\log n} 2^{\beta \log n}.
$$
Using the previous notation, we have
$$
    {C(n) \over R(n)} \ge {\left(\ds{\alpha \over 2}\right)^{\beta''\log n}\left(\ds{c \over \beta''-\beta}\right)^{(\beta''-\beta)\log n}
    \over n 2^{\beta \log n}},
$$
which tends to infinity if $c$ is large enough.
\end{proof}

\section{Conclusion and open problems}

Our work suggests several conjectures and open problems.

\begin{enumerate}
  \item We conjecture that the lower bound
  $$
    \D_n \ge c  {\log n \over \log^{(\ell+1)} n}
    $$
    for the class $\ex(C_{2\ell+1})$ is of the right order of magnitude. The proofs for $\ex(C_5)$ and $\ex(C_7)$ seem difficult to adapt for arbitrary $\ell$.

  \item We conjecture that the asymptotic behaviour of $\D_n$ is the same for $\ex(C_{2\ell})$ as for $\ex(C_{2\ell-1})$. We have shown this is the case for $\ell=2$ and $\ell=3$.
  \item We conjecture an upper bound of the form
  $$
    \D_n \le c \log n
    $$
    for the class $\ex(H_1,\dots,H_k)$, whenever the $H_i$ are 2-connected.
    Examples show that this is not true for arbitrary $H$ (see the discussion below).
    Using analytic methods, this upper bound can be proved for so-called subcritical classes of graphs (see~\cite{subcritical}), which include outerplanar and series-parallel graphs.

  \item Which are the possible orders of magnitude of $\D_n$ when forbidding a 2-connected graph? Assuming the truth of the conjecture in item 1, are there other possibilities besides $\log n$ and $\log n / \log^{(k+1)} n$?
  \item Which are the possible orders of magnitude of $\D_n$ for arbitrary minor-closed classes of graphs? Besides those discussed above, examples show that it can be constant (forbidding a star) and it can be linear (forbidding two disjoint triangles). The last statement follows from~\cite{KM}, where it is shown that the class $\ex(C_3 \cup C_3)$ is asymptotically the same as the class of  graphs $G$ having a vertex $v$ such that $G-v$ is a forest.

\item Is it true that if $H$ consists of a cycle and some chords, all of them incident to the same vertex, then  $\Delta_n = o(\log n)$ holds  in $\ex(H)$ w.h.p.?
    These are the 2-connected graphs that are a minor of some fan $F_n$, so that the proof of the first part in Theorem~\ref{th:lower} does not hold.

  \item Prove an upper bound $\D_n \le c \log n$ for series-parallel graphs without using the analysis of generating functions as in~\cite{DGN3}.
      More generally, prove such a bound for graphs of bounded tree-width (series-parallel graphs are those with tree-width at most two). In the last section  we have provided such a proof  for outerplanar graphs.

\end{enumerate}

\end{document}